\documentclass[reqno, 12pt, a4letter]{amsart}
\usepackage{amsmath,amsxtra,amssymb,latexsym, amscd,amsthm,pb-diagram}

\usepackage{epsfig,cases}
\usepackage{graphics}
\usepackage{color}
\usepackage{ulem}
\usepackage[all]{xy}
\usepackage{geometry}

\numberwithin{equation}{section}
\theoremstyle{plain}
\newtheorem{thm}{Theorem}[section]

\newtheorem{lemma}[thm]{Lemma}

\theoremstyle{definition}
\newtheorem{dfn}[thm]{Definition}

\theoremstyle{remark}
\newtheorem{rem}[thm]{Remark}

\newcommand{\C}{\mathbb C}

\newcommand{\R}{\mathbb R}

\newcommand{\Z}{\mathbb Z}

\def\dim{\operatorname{dim}}
\def\grad{\operatorname{grad}}

\def\min{\operatorname{min}}

\def\supp{\operatorname{supp}}
\def\EDdeg{\operatorname{EDdeg}}

\def\ees{{\accent"5E e}\kern-.385em\raise.2ex\hbox{\char'23}\kern-.08em}
\def\EES{{\accent"5E E}\kern-.5em\raise.8ex\hbox{\char'23 }}
\def\ow{o\kern-.42em\raise.82ex\hbox{
\vrule width .12em height .0ex depth .075ex \kern-0.16em \char'56}\kern-.07em}
\def\OW{O\kern-.460em\raise1.36ex\hbox{
\vrule width .13em height .0ex depth .075ex \kern-0.16em \char'56}\kern-.07em}

\def\DD{D\kern-.7em\raise0.4ex\hbox{\char '55}\kern.33em}

\begin{document}
\title[Euclidean distance degree of complete intersections]{Euclidean distance degree of complete intersections via Newton polytopes}

\author[T. T. Nguyen]{Nguyen Tat Thang}
\address{Institute of Mathematics, Vietnam Academy of Science and Technology, 18 Hoang Quoc Viet road, Cau Giay district, Hanoi, Vietnam}
\email{ntthang@math.ac.vn}

\author[T. T. Pham]{Pham Thu Thuy}
\address{Department of Mathematics - Faculty of Fundamental Sciences,
Posts and Telecommunications Institute of Technology,
Km10 Nguyen Trai Road, Ha Dong District, Hanoi, Vietnam}
\email{thuyphamun@gmail.com}

\keywords{Euclidean distance degree, Newton polytopes, Mixed volume, Critical point, Tangent space}

\begin{abstract}
		In this note, we consider a complete intersection $X=\{x\in \mathbb{R}^n : f_1(x)= \ldots = f_m(x)=0\}, n>m$ and study its Euclidean distance degree in terms of the mixed volume of the Newton polytopes. We show that if the Newton polytopes of $f_j,j=1,\ldots, m$ contain the origin then when these polynomials are generic with respect to their Newton polytopes, the Euclidean distance degree of $X$ can be computed in terms of the mixed volume of Newton polytopes associated to $f_j$.  This is a generalization for the result by P. Breiding, F. Sottile and J. Woodcock in case $m=1$.
\end{abstract}
\maketitle	

\section{Introduction}

Many models in data science or mechanical engineering can be represented as a real algebraic set, leading to the need to solve the problem of finding the nearest point.

Nearest point problem: In $\mathbb{R}^n$ given the algebraic set $\mathrm{X}$ and a point $u$, find a point $u^*$ of $\mathrm{X}$ that minimizes the squared Euclidean distance function $d_u(x)=\Sigma\left(x_i- u_i\right)^2$ from the given point $u$.

That optimization problem arises
in many applications. For instance, if $u$ is a noisy sample from $X$, where the error model is a standard Gaussian in $\mathbb{R}^n$ then $u^*$
is the maximum likelihood estimate for $u$. Also, in \cite{S} Seidenberg  observed that if $X$ is nonempty, then it contains a solution to the problem above. He
used this observation in an algorithm for deciding if $X$ is empty.

One approach to the "nearest point problem" is to find and examine all complex critical points of $d_u$. More precisely, we consider the complexification $X_{\C}$ of $X$ (i.e. the set of complex solutions in $\mathbb{C}^n$ of the defining equations of $X$) and examine
all complex critical points of the squared distance function $d_u(x) = \Sigma\left(x_i- u_i\right)^2$ on $X_{\C}$ which are not singular points of $X_{\C}$. The number of such critical points are constant on a dense subset of $u\in \mathbb{R}^n$ and is called the Euclidean distance degree of the set $\mathrm{X}$, denoted by $\mathrm{EDdeg}(\mathrm{X})$.

The Euclidean distance degree of $X$ gives us an algebraic quantity that evaluates the complexity of the nearest point problem for $X$.  This notion was introduced in \cite{DHOST}, and has since been extensively studied in areas like computer vision, biology, chemical reaction networks, engineering, numerical algebraic geometry, data science, ...

There are many study on the Euclidean distance degree. In the paper \cite{DHOST}, the authors introduced the EDdeg under the algebraic geometry view of point.  In \cite{AH}, the authors gave a formula of EDdeg of projective varieties involving characteristic classes. In \cite{MRW},  the EDdeg of a complex affine variety is given in terms of Euler characteristic. Another observation is considered in \cite{BSW} where the authors gave a combinatoric formula for the EDdeg of a hypersurface in terms of the mixed volume of the related Newton polytopes. In a recent paper \cite{LMR}, instead of the distance function, the authors considered a general polynomial function $f_0$ on an algebraic set $X:= \{x\in \mathbb{C}^n : f_1(x)= \ldots = f_m(x)=0\}, n>m$ and  proved that if the tupe of supports of $f_i$ is strongly admissible (see \cite{LMR} for detail) then the number of critical points of the restriction of $f_0$ on $X$ is computed in terms of the mixed volume of the Newton polytopes of $f_i$ and related functions. That is a very generalized result of the one in \cite{BSW}.

The aim of this paper is to give another generalization to the formula in \cite{BSW} for complete intersections. Here, we use different method, with the setting of the distance function, we reduce the assumption as in \cite{LMR}. More precisely, let $X=\{x\in \mathbb{R}^n : f_1(x)= \ldots = f_m(x)=0\}, n>m$ where $f_1, \ldots, f_m$ are polynomials in $n$ variables such that $\dim X= n-m$. For $u= (u_1, \ldots, u_n)\in \mathbb{R}^n$, denote by $P_1,\ldots,P_m\subset \mathbb{R}^n$ the Newton polytopes of $f_1, \ldots,f_m$ and $P'_i\subset \mathbb{R}^{n+m}, i=1, \ldots, n$ the Newton polytopes of $x_i-u_i+\sum_{j=1}^m \lambda_j \partial_i f_j(x)\in \mathbb{R}[x,\lambda], \lambda= (\lambda_1, \ldots, \lambda_m)$. The main results in this paper are the followings.

\begin{thm} \label{Th1}
 Assume that $X$ is a complete intersection whose no component is contained in a coordinate hyperplane.  Then
\begin{equation}\label{inequality1}
\operatorname{EDdeg}(X) \leq\operatorname{MV}\left(P_1,\ldots,P_m, P'_1, \ldots, P'_n\right),
\end{equation}
 where $\operatorname{MV}$ denotes the mixed volume of $n+m$ polytopes in $\mathbb{R}^{n+m}$. 
\end{thm}

\begin{thm} \label{Th2}
 Assume that the supports of $f_1, \ldots,f_m$ contain the origin $0\in \mathbb{R}^n$. Then, when $f_1, \ldots,f_m$ are generic with respect to their Newton polytopes, we have
\begin{equation}\label{equality2}
\operatorname{EDdeg}(X) = \operatorname{MV}\left(P_1,\ldots,P_m, P'_1, \ldots, P'_n\right).
\end{equation}
 
\end{thm}

The proofs of these results are given in the last section, while the next section devotes to recall some notions and results concerning Newton polytopes.

\section{Preliminaries}

\subsection{Khovanskii non-degeneracy}
Let $h:\mathbb{C}^n \rightarrow \mathbb{C}$ be a polynomial function. Assume that $h$ is written as $h(x)=\sum_{\alpha\in \mathbb{N}^n} a_\alpha x^\alpha$, where $x^{\alpha}=x_1^{\alpha_1}\ldots x_n^{\alpha_n}$ for $\alpha= (\alpha_1, \ldots, \alpha_n)$. The {\it support} of $h$ is defined as the set of those $\alpha \in \mathbb{N}^n$ such that $a_\alpha \neq 0$, denoted by $\operatorname{supp}(h)$.
The {\it Newton polytope} of $h$, denoted by $\Gamma(h)$, is defined as the convex hull in $\mathbb{R}^n$ of the support of $h$.
The {\it Newton polyhedron} of $h$, denoted by $\Gamma_{+}(h)$, is by definition the  boundary of $\Gamma(h)$.  For each face $\Delta$ of $\Gamma_{+}(h)$, we will denote by $h_{\Delta}$ the polynomial $\sum_{\alpha \in \Delta} a_\alpha x^\alpha$; if $\Delta \cap \operatorname{supp}(h)=\emptyset$ we let $h_{\Delta}:=0$.
We say that a subset $\Gamma$
of $\mathbb{R}^n$ is a {\it Newton polytope} if there is a finite subset $A\subset \mathbb{N}^n$ such that $\Gamma$ is equal
to the convex hull of $A$. 

 Given polytopes $\Gamma_1, \ldots, \Gamma_m$ in $\mathbb{R}^m$, we write $\operatorname{MV}(\Gamma_1, \ldots, \Gamma_m)$ for their mixed volume. The notion of mixed volume
was defined by Minkowski; see its definition and properties in \cite[Sect. IV.3]{Ewald}.

Given a nonzero vector $q \in \mathbb{R}^n$, we define
$$
\begin{aligned}
& d\left(q, \Gamma(h)\right):=\min \left\{\langle q, \alpha\rangle: \alpha \in \Gamma(h)\right\} \\
& \Delta\left(q, \Gamma(h)\right):=\left\{\alpha \in \Gamma(h):\langle q, \alpha\rangle=d\left(q, \Gamma(h)\right)\right\} .
\end{aligned}
$$
It is easy to check that for each nonzero vector $q \in \mathbb{R}^n, \Delta:=\Delta\left(q, \Gamma(h)\right)$ is a closed face of $\Gamma(h)$. Conversely, if $\Delta$ is a closed face of $\Gamma(h)$ then there exists a nonzero vector $q \in \mathbb{R}^n$ such that $\Delta=\Delta\left(q, \Gamma(h)\right)$. We also denote by $h_q$ the polynomial $h_{\Delta}$.

\begin{rem}  Let $\Delta:=\Delta\left(q, \Gamma(h)\right)$ for some nonzero vector $q:=\left(q_1, \ldots, q_n\right) \in \mathbb{R}^n$.
By definition, $h_{\Delta}=\sum_{\alpha \in \Delta} a_\alpha x^\alpha$ is a weighted homogeneous polynomial of type $\left(q, d:=d\left(q, \Gamma(h)\right)\right)$, i.e., we have for all $t$ and all $x \in \mathbb{C}^n$,
$$
h_{\Delta}\left(t^{q_1} x_1, \ldots, t^{q_n} x_n\right)=t^d h_{\Delta}\left(x_1, \ldots, x_n\right) .
$$
This implies the Euler relation
\begin{equation}\label{eulerrelation}
\sum_{i=1}^n q_i x_i \frac{\partial h_{\Delta}}{\partial x_i}(x)=d h_{\Delta}(x).
\end{equation}
\end{rem}

Now, consider $m$ polynomial functions $f_1, \ldots, f_m:\mathbb{C}^n \rightarrow \mathbb{C}, n>m$. The following definition of non-degeneracy is inspired from the work of Kouchnirenko \cite{Kouchnirenko}, where the case $m=1$ was considered (see  \cite{Khovanskii,Oka1,Oka2}).

\begin{dfn}{\rm 
 We say that the $m$-tuple $f:= (f_1,\ldots, f_m)$ is (Khovanskii) non-degenerate if, for any nonzero vector $q \in \mathbb{R}^n$ we have
$$
f_q^{-1}(0) \cap\left\{x \in \mathbb{C}^n: \operatorname{rank}\left(J f_q(x)\right)<m\right\} \subset\left\{x \in \mathbb{C}^n: x_1 \ldots x_n=0\right\},
$$
where $f_q$ denotes the map $\left((f_1)_q, \cdots, (f_m)_q\right)$ and $J f_q$ denotes the Jacobian matrix of the map  $f_q: \mathbb{C}^n \rightarrow \mathbb{C}^m$.}
\end{dfn}

Let $A$ be a compact subset of $\mathbb{R}^n$ such that $A\cap \mathbb{N}^n\neq \emptyset$. Let us define
$$\mathcal{P}(A) = \{h\in \mathbb{C}[x]: supp(h)\subseteq A \cap \mathbb{N}^n\}.$$
Since $A$ is compact, $\mathcal{P}(A)$ is a finite-dimensional complex vector space.
Let $V$ be a complex vector space of finite dimension. We say that a given property is generic
in $V$ when there exists a Zariski-open set $U\subseteq V$ such that any element $v\in  U$ satisfies the said
property.

The non-degeneracy of  polynomial maps is generic (see \cite[page 291]{Khovanskii} or \cite[page 263]{Oka2}) as stated below.

\begin{lemma}\label{nondegenerategeneric}
Let $\Gamma_1, \ldots, \Gamma_m$ be Newton polytopes in $\mathbb{R}^n$. Let $\mathcal{P}$ be the product vector space defined by 
$$\mathcal{P}:= \mathcal{P}(\Gamma_1)\times \ldots \times \mathcal{P}(\Gamma_m).$$
Then the property that a map $f=(f_1, \ldots, f_m)$ in $\mathcal{P}$ is non-degenerate is generic in $\mathcal{P}$.
\end{lemma}

\subsection{Bernstein's theorems}
The number of isolated solutions of a system of polynomial equations is estimated in terms of the mixed volume of Newton polytopes of the polynomials in the system.

\begin{thm} {\bf (see \cite{Bernstein})}\label{Ber1}
Let $g_1, \ldots, g_m \in \mathbb{C}\left[x_1, \ldots, x_m\right]$ be $m$ polynomials with Newton polytopes $\Gamma_1, \ldots, \Gamma_m$. Let $\# \mathcal{V}_{\mathbb{C}^{\times}}\left(g_1, \ldots, g_m\right)$ be the number of isolated solutions to $g_1=\cdots=g_m=0$ in $\left(\mathbb{C}^{\times}\right)^m$ , counted by their algebraic multiplicities. Then
$$
\# \mathcal{V}_{\mathbb{C}^{\times}}\left(g_1, \ldots, g_m\right) \leq \operatorname{MV}\left(\Gamma_1, \ldots, \Gamma_m\right)
$$
and the inequality becomes an equality when each $g_i$ is general given its support. 
\end{thm}

We also need other version of Bernstein's theorem.

\begin{thm}{\rm (Bernstein's Other Theorem, see \cite[Theorem B]{Bernstein})} \label{Ber2}
Let $G=\left(g_1, \ldots, g_m\right)$ be a system of Laurent polynomials in variables $x_1, \ldots, x_m$. For each $1 \leq i \leq m$, let $\Gamma_i$ be the Newton polytope of $g_i$. Then
$$
\# \mathcal{V}_{\mathbb{C}^{\times}}\left(g_1, \ldots, g_m\right)<\operatorname{MV}\left(\Gamma_1, \ldots, \Gamma_m\right)
$$
if and only if there is a nonzero vector $w \in \mathbb{Z}^m$ such that the facial system $G_w$ has a solution in $\left(\mathbb{C}^{\times}\right)^m$. Otherwise, $\# \mathcal{V}_{\mathbb{C}^{\times}}\left(g_1, \ldots, g_m\right)$ is equal to $\operatorname{MV}\left(\Gamma_1, \ldots, \Gamma_m\right)$.
\end{thm}

\section{Proof of the main results}
Let $f_1, \ldots, f_m\in \mathbb{R}[x]$ be polynomials in $n$ variables $x_1, \ldots, x_n (n>m)$ such that the algebraic set $X:=\{x\in \mathbb{R}^n : f_1(x)=\ldots = f_m(x)=0\}$ is of dimension $n-m$. Then its complexification $X_{\C}=\{x\in \mathbb{C}^n : f_1(x)=\ldots = f_m(x)=0\}$ is also of dimension  $n-m$.  We denote by $Sing(X_{\C})$ the set of singular points of $X_{\C}$ and $(X_{\C})_{\text{reg}}:=X_{\C}\setminus Sing(X_{\C})$ the set of regular points of $X_{\C}$.

For a point $u=\left(u_1, \ldots, u_n\right) \in \mathbb{C}^n$, we consider the restriction of the distance function
$$\begin{aligned} d_u: X_{\C} & \longrightarrow \mathbb{C} \\ x & \longmapsto\|x-u\|^2=\sum_{i=1}^n\left(x_i-u_i\right)^2.\end{aligned}$$
The set of critical points of $d_u$ which do not belong to $Sing(X_{\C})$ is
\begin{align*}
\Sigma_u=\left\{x \in (X_{\C})_{\text{reg}}: x-u+\sum_{j=1}^m \lambda_j \operatorname{grad} f_j(x)=0\quad\textrm{for some}\quad \lambda_1, \ldots, \lambda_m \in \mathbb{C} \right\}. 
\end{align*}
By definition, $\mathrm{EDdeg}(\mathrm{X})$ is equal to the number of points in $\Sigma_u$ for $u$ is generic.

We denote by $\Sigma$ the following set.
$$
\Sigma=\left\{(x,\lambda)\in (X_{\C})_{\text{reg}} \times \mathbb{C}^m: x-u+\sum_{j=1}^m \lambda_j \operatorname{grad} f_j(x)=0\right\}.
$$
It is easy to check that
 $$\# \Sigma_u=\# \Sigma.$$
 
Also, we consider the set $M$ of points $(u, x, \lambda) \in \mathbb{C}^n \times (\mathbb{C}^n\setminus Sing(X_{\C})) \times \mathbb{C}^m$ satisfying the following system:
\begin{equation}\label{1.2}
f_1(x)=\ldots = f_m(x)=0\quad\textrm{and}\quad x-u+\sum_{j=1}^m \lambda_j \operatorname{grad} f_j(x)=0.
\end{equation}
One can check that $M$ is a complex manifold with $\operatorname{dim} M=n$.
Consider the projection
$$
\begin{aligned}
\pi_2: \hspace{0.8cm} M &\longrightarrow (X_{\C})_{\text{reg}} \\
(u, x, \lambda) &\mapsto x.
\end{aligned}
$$
For each $x_0 \in (X_{\C})_{\text{reg}},$
$$\pi_2^{-1}\left(x_0\right)=\left\{(u, \lambda) \in \mathbb{C}^n \times \mathbb{C}^m: u-x_0=\sum_{j=1}^m \lambda_j \operatorname{grad} f_j\left(x_0\right)\right\}.$$ It is easy to see that $\pi_2^{-1}\left(x_0\right)$ is isomorphic to $\mathbb{C}^m$. Thus, one can show that $\pi_2: M \longrightarrow (X_{\C})_{\text{reg}}$ is a trivial fibration with fibers $\mathbb{C}^m$.

Denote by $\pi_1: M\longrightarrow \mathbb{C}^n$  the projection $(u,x,\lambda)\mapsto u$, then $\pi_1$  is dominant. The fiber $\pi_1^{-1}(u_0)$ over $u_0$ is the set of $(x,\lambda)\in (\mathbb{C}^n\setminus Sing(X)) \times \mathbb{C}^m$ satisfying the system \eqref{1.2}.
Since $\operatorname{dim} M=n$ the generic fiber $ \pi_1^{-1}(u)$ of $\pi_1$ has dimension $n-n=0$. Therefore $
\# \pi_1^{-1}(u)= \text{const}$
when $u$ is not a critical point of $\pi_1$.
Hence, 
$$\# \pi_1^{-1}(u)=\# \Sigma_u = \mathrm{EDdeg}(X).$$

Now, let consider 
$$
\begin{aligned}
\pi_3: \hspace{0.8cm} M &\longrightarrow \mathbb{C}^m\\
(u, x, \lambda) &\mapsto \lambda \text {, }
\end{aligned}
$$
then for each $\lambda\in \C^m$ we have
$$\pi_3^{-1}(\lambda)=\left\{(u, x, \lambda) \in \mathbb{C}^n \times (X_{\C})_{\text{reg}} \times \mathbb{C}^m \mid  u-x=\sum_{j=1}^m \lambda_j \operatorname{grad} f_j(x)\right\}.$$
It is easy to see that $\pi_3^{-1}(\lambda)$ is isomorphic to $(X_{\C})_{\text{reg}}$, so $\pi_3: M \longrightarrow \mathbb{C}^m$ is a trivial fibration with fiber $(X_{\C})_{\text{reg}}$.
Thus, for any subset $U$ of $\C^m$ one has
 $$\pi_3^{-1}(U) \simeq (X_{\C})_{\text{reg}} \times U.$$
Let $Y \subset \mathbb{C}^m$ be the union of coordinate hyperplanes. Then $\operatorname{dim} Y=m-1$ and  
$$\operatorname{dim} \pi_3^{-1}\left(Y\right)=\operatorname{dim} Y+\operatorname{dim} X_{\C}=(m-1)+(n-m)=n-1.$$
Denote $$\gamma=\pi_1\left(\pi_3^{-1}\left(Y\right)\right) \subset \mathbb{C}^n,$$ then $$\operatorname{dim} \gamma \leq n-1.$$

\begin{proof}[Proof of Theorem \ref{Th1}]
If $X$ does not have any irreducible component contained in a coordinate hyperplane of $\R^n$ then its complexification $X_{\C}$ also has no component contained in a coordinate hyperplane of $\C^n$.
Denote by $Y^{'}$ the intersection of $X_{\C}$ with the coordinate hyperplanes of $\C^n$. Then $\dim Y^{'}< \dim X_{\C}=n-m$ (otherwise $Y^{'}$ will be an union of some irreducible components of $X_{\C}$).

Let denote $$\gamma^{'}:= \pi_1\left(\pi_2^{-1}(Y^{'})\right).$$
We have
$$
\operatorname{dim} \pi_2^{-1}(Y^{'})=\operatorname{dim} Y^{'}+\operatorname{dim} \mathbb{C}^m<n-m+m=n .
$$
Thus, $\operatorname{dim} \gamma^{'}<n$.
Choose $u_0 \in \mathbb{C}^n \backslash\left(\gamma \cup \gamma^{\prime}\right)$, we have
$$
\pi_1^{-1}\left(u_0\right)=\left\{\left(u_0, x, \lambda\right) \in \mathbb{C}^n \times(\left(\mathbb{C}^\times\right)^n\setminus Sing(X_{\C})) \times\left(\mathbb{C}^\times\right)^m \mid (u_0, x,\lambda) \text { satisfies } \eqref{1.2}\right\} .
$$

Therefore $\operatorname{EDdeg}(X)$ is not larger than the number of solutions in $\left(\mathbb{C}^\times\right)^n \times\left(\mathbb{C}^\times\right)^m$ of the system \eqref{1.2}. The inequality \eqref{inequality1} is then followed by applying Theorem \ref{Ber1}.

\end{proof}

Before proving Theorem \ref{Th2}, we need the following properties of algebraic sets defined by polynomials with generic coefficients.

\begin{lemma}\label{lm1}
Let $X=\{x\in \mathbb{R}^n : f_1(x)= \ldots = f_m(x)=0\}, n>m$ where $f_1, \ldots, f_m$ are non-constant polynomials such that their Newton polytopes $P_1, \ldots, P_m$ contain the origin. Then, if $f_j, j=1, \ldots, m$ are generic with respect to their Newton polytopes $P_j$, the complexification $X_{\C}$ of $X$ is smooth and irreducible. 
\end{lemma}

\begin{proof}
We write each polynomial $f_j$ as  
$$f_j(x)= \sum_{\alpha\in \supp f_j} a^j_{\alpha}x^{\alpha}.$$
By fixing an order on the set of all monomials of $f_j, j=1, \ldots, m$, we have the space $\C^N$ of all coefficients of all $f_j,j=1,\ldots, m.$ Denote
$$\Sigma:=\left\lbrace(x, a^j_{\alpha}) : \sum_{\alpha\in \supp f_j} a^j_{\alpha}x^{\alpha}=0, j=1,\ldots, m\right\rbrace.$$
Because of the assumption $f_j(0)\neq 0$ for all $j$, one can check that  $\Sigma$ is a smooth complex algebraic set of dimension $n+N-m$, in fact $\Sigma$ is isomorphic to $\C^{n+N-m}$. In particular $\Sigma$ is irreducible.

Consider the projection
$$\pi: \Sigma\longrightarrow \C^N, (x, a^j_{\alpha})\mapsto (a^j_{\alpha}).$$
By Sard's Theorem we obtain that for generic $(a^j_{\alpha})\in \C^N$, $\pi^{-1}(a^j_{\alpha})$ is smooth. That means if the  coefficients of $f_j, j=1, \ldots, m$ are generic then $X_{\C}$ is smooth.

On the other hand, we can assume that $\pi$ is dominant (otherwise we can replace the target by the Zariski closure of the image of $\pi$). It is known that (see \cite{Th,Ve}) there exists a proper algebraic subset $B$ of the target $\C^N$ of $\pi$ such that
$$\pi|_{\Sigma\setminus \pi^{-1}(B)}: \Sigma\setminus \pi^{-1}(B)\longrightarrow \C^N\setminus B$$
is a locally trivial fibration.

Since $\dim B<N$ and $\Sigma$ is connected, it follows from   \cite[p.68]{Mu} that both $\Sigma\setminus \pi^{-1}(B)$ and $\C^N\setminus B$ are connected. This means $\pi|_{\Sigma\setminus \pi^{-1}(B)}$ is a locally trivial fibration with connected base and connected total space, therefore its fibers are also connected. Hence, the generic fiber of $\pi$ is a smooth complex algebraic set which is connected, then it is irreducible. The proof is complete.

\end{proof}

\begin{lemma}\label{lm2}
Let $X=\{x\in \mathbb{R}^n : f_1(x)= \ldots = f_m(x)=0\}, n>m$ where $f_1, \ldots, f_m$ are non-constant polynomials such that their Newton polytopes $P_1, \ldots, P_m$ contain the origin. Then, if $f_j, j=1, \ldots, m$ are generic with respect to their Newton polytopes $P_j$, the complexification $X_{\C}$ of $X$ is not contained in a coordinate hyperplane of $\C^n$. 
\end{lemma}

\begin{proof}
It follows from Lemma \ref{lm1} that if $f_j, j=1, \ldots, m$ are generic with respect to their Newton polytopes then $X_{\C}$ is irreducible. Assume that $X_{\C}$ is  contained in some coordinate hyperplane $\{x_i=0\}$ of $\C^n$, it deduces that $x_i$ belongs to the ideal $\langle f_j, j=1,\ldots, m \rangle$ generated by $f_j, j=1,\ldots, m.$ Since $f_j(0)\neq 0$ for all $j$, by choosing  $f_j(0)$ generic  we get the contradiction. The proof is complete.
\end{proof}

\begin{proof}[Proof of Theorem \ref{Th2}]
It implies from Lemma \ref{lm1} and Lemma \ref{lm2} that if $f_j,j=1,\ldots, m$ are generic with respect to their Newton polytopes  $X_{\C}$ is smooth and is not contained in any coordinate hyperplane. Hence, it follows from the proof of the Theorem \ref{Th1} that the Euclidean distance degree $\EDdeg(X)$ of $X$  is equal to the number of points $(x, \lambda)\in (\C^{\times})^n\times (\C^{\times})^m$ satisfying the system \eqref{1.2} for a generic point $u\in\C^n$.
Therefore, by applying the Bernstein's Other Theorem \ref{Ber2} and the following theorem, we get the equality \eqref{equality2}.

\end{proof}

\begin{thm}\label{facefunction} Suppose that $f_1, \ldots, f_m$ are generic with respect to their Newton polytopes $P_1, \ldots, P_m$ with $0\in P_j$ for all $j=1, \ldots, m$ and $u\in \C^n$ is generic. Then, for all nonzero $w\in \Z^{n+m}$, the system of face functions 
\begin{equation}
(f_1)_w=\cdots= (f_m)_w=0, \left(x-u+\sum_{j=1}^m \lambda_j \operatorname{grad} f_j(x)\right)_w=0
\end{equation}
 of the system \eqref{1.2} has no solution in $(\C^{\times})^n\times (\C^{\times})^m$.
\end{thm}

\begin{proof}

Let $w=\left(w_1, \ldots, w_n, v_1, \ldots, v_m\right) \neq  0 \in \mathbb{R}^n \times \mathbb{R}^m$ arbitrary, we need to show that when $f_j$'s are generic with respect to their Newton polytopes and the vector $u=(u_1, \ldots,u_n)\in \C^n$ is generic, the system
\begin{equation}\label{facesystem}
\left\{\begin{array}{l}
\left(f_1\right)_w(x)=\ldots=\left(f_m\right)_w(x)=0 \\
\left(x_i-u_i+\sum_{j=1}^m \lambda_j \partial_i\left[f_j(x)\right]\right)_w=0, i=\overline{1,n},
\end{array}\right.
\end{equation}
has no solution in $\left(\mathbb{C}^\times\right)^n \times\left(\mathbb{C}^\times\right)^m$.

For $j\in\{1,\ldots,  m\} $ and $i\in\{1,\ldots,  n\}$, let $P_j$ be the Newton polytopes of $f_j(x) $, and
$P_j^i$ be the Newton polytopes of $\partial_i[f_j]:=\frac{\partial}{\partial x_i}\left(f_j\right)$. We regard $P_j$ and $P_j^i$ as polytopes in  $\mathbb{R}^n \times \mathbb{R}^m$. Set
$h_j=\min _{a \in P_j} \langle w, a\rangle, h_j^i=\min _{a \in P_j^i} \langle w, a\rangle.$ Because the Newton polytope $P_j$ of $f_j$ contains the origin, we see that 
\begin{equation}\label{ordernegative}
h_j\leq 0.
\end{equation}
Recall that $ \left(P_j\right)_w:= \Delta(w, P_j)$  and $\left(P_j^i\right)_w:= \Delta(w, P^i_j)$  are the faces  of $P_j$ and $P_j^i$ defined by the vector $w$.
 Also,  put 
$e_i:=\min_{j\in\{1,\ldots, m\}} \left\{v_j+h_j^i\right\}$
and 
$$S_i=\lbrace k \in\{1, \ldots, m\} \mid v_k+h_k^i=e_i\rbrace.$$
We also set $e:=\min _{j\in\{1,\ldots, m\}}\left\{v_j+h_j\right\}$ and
 $S:=\left\{k \in\{1, \ldots, m\} \mid v_k+h_k=e\right\}$.

There are $7$ cases for the face functions in \eqref{facefunction}:

$ H_i:=\left(x_i-u_i+\sum_{j=1}^m \lambda_j \partial_i\left[f_j(x)\right]\right)_w= $

\begin{numcases}{}
-u_i & $\text { if } 0<w_i ; 0<e_i$,\label{1}
\\ x_i & $\text { if } w_i<\min \left\{0, e_i\right\}$, \label{2}
\\ \sum_{k \in S_i} \lambda_k\left[\partial_i\left(f_k\right)\right]_w & $\text { if } e_i<\min \left\{0, w_i\right\} $,\label{3}
\\ x_i-u_i & $\text { if } w_i=0<e_i$, \label{4}
\\ -u_i+\sum_{k \in S_i} \lambda_k\left[\partial_i\left(f_k\right)\right]_w & $\text { if } 0=e_i<w_i $,\label{5}
\\ x_i+\sum_{k \in S_i} \lambda_k\left[\partial_i\left(f_k\right)\right]_w & $\text { if } w_i=e_i<0$,\label{6}
\\ x_i-u_i+\sum_{k \in S_i} \lambda_k\left[\partial_i\left(f_k\right)\right]_w & $\text { if } w_i=e_i=0$.\label{7}
\end{numcases}
Assume by contraposition that the system of face functions \eqref{facesystem} has solution in $(\C^{\times})^n\times (\C^{\times})^m$. Then by choosing $u_i\neq 0$ for all $i$ we can skip the cases \eqref{1} and \eqref{2}.

For each $i \in\{1, \ldots, n\}, j \in\{1, \ldots, m\}$, according to \cite[Lemma 8]{BSW}, we have 
\begin{equation}\label{inequalityhij}
h_j^i\geq h_j-w_i,
\end{equation}
 then 
$$v_j+h_j^i\geq v_j+h_j-w_i \geq e-w_i,$$
the equality occurs if and only if $\partial_i\left(f_j\right)_w \neq 0$ and  $j \in S$. In that case, we also have $e_i=e-w_i$. 
Otherwise, if $\partial_i\left(f_j\right)_w=0$ then $h_j^i>h_j-w_i$ and $v_j+h_j>e-w_i.$

There are two cases.

\textbf{Case 1}: There exists $j \in\{1, \ldots, m\}$ such that $\partial_i\left(f_j\right)_w=0$ for all $i\in \{1,\ldots, n\}$. This means $\left(f_j\right)_w$ is the constant term of $f_j$. Hence when $f_j$ is general, we get $\left(f_j\right)_w \neq 0$. Thus, the system of face functions \eqref{facesystem} has no solution in $(\C^{\times})^n\times (\C^{\times})^m$.

\textbf{Case 2}: For each $j \in\{1, \ldots, m\},\left(f_j\right)_\omega$ is not a constant, this means there exists $i$ such that $\partial_i\left(f_j\right)_w \neq 0$.
Let $I:=\left\{i \in\{1, \ldots, n\} \mid \exists j \in S: \partial_i\left(f_j\right)_w \neq 0\right\}$, it is obvious that $I \neq \emptyset$. Moreover, it implies from the above observation that for any $i\in I$ one has 
\begin{equation}\label{formulare}
e_i=e-w_i.
\end{equation}
 Also, for  $k \in S_i$ we have $$e_i=v_k+h_k^i=\min _{j\in \{1,\ldots,  m\}}\left(v_j+h_j^i\right)=e-w_i,$$ 
so
$h_k^i=e-w_i-v_k=h_k-w_i$, by \cite[Lemma 8]{BSW}, we get $\partial_i\left(f_k\right)_w \neq 0$ and $\left[\partial_i\left(f_k\right)\right]_w=\partial_i\left(f_k\right)_w$. Hence, we have $S_i=\{j\in S: \partial_i\left(f_j\right)_w \neq 0\}$. 

\textit{Case 2.1}: There exists $s \in I$ such that $w_s<0$.
Let $s \in I$ be such an element, it means $w_s<0$. Then, for such element $s$ either \eqref{3} or \eqref{6} occurs. In particular $e_s<0,$ it implies that $ e<w_s<0.$ Therefore, for any $i \in I$, we obtain that $e_i=e-w_i<-w_i.$ So if $w_i\geq 0$, we also have $e_i<0$. Thus $e_i<0$ for all $i \in I$. This means, for any $i \in I$ either \eqref{3} or \eqref{6} occurs.

Let
\begin{equation*}
 K:=\left\{i \in I \mid e_i<\min \left\{0, w_i\right\}\right\}, 
 M:=\left\{i \in I \mid \quad w_i=e_i<0\right\}.
\end{equation*}
For $i\in M$, we have $e_i=e-w_i$ and $w_i=e_i$ so $w_i=e/2.$
If $M=\emptyset$ then $K=I$.
This mean, for any $i \in I$ then 
$$\left(x_i-u_i+\sum_{j=1}^m \lambda_j \partial_i\left(f_j\right)\right)_w
=\sum_{k \in S_i} \lambda_k\left[\partial_i\left(f_k\right)\right]_w=\sum_{k \in S} \lambda_k \partial_i\left(f_k\right)_w.
$$
However, by the definition of $I$, if  $l\notin I$, for any $k\in S$ we get $\partial_l\left(f_k\right)_w=0$. In other words,  $\left(f_k\right)_w \in \mathbb{C}\left[x_I\right]$ for all $k \in S$, where $x_I=(x_i, i\in I)$. Hence,  the system \eqref{facesystem} has solution in $(\C^{\times})^n\times (\C^{\times})^m$ implying that the tupe $(f_k, k\in S)$ is degenerate, this is not the case if we choose the polynomials $f_j$ generic with respect to their Newton polytopes (by Lemma \ref{nondegenerategeneric}).

If $M \neq \emptyset$,  for $i \in K$, we have
\begin{equation}\label{setK}
H_i=\left(x_i-u_i+\sum_{j=1}^m \lambda_j\left[\partial_i\left(f_j\right)\right]\right)_w=\sum_{k \in S_i} \lambda_k\left[\partial_i\left(f_k\right)\right]_w=\sum_{k \in S} \lambda_k \partial_i\left(f_k\right)_w
\end{equation}
and for $i \in M$ we have
\begin{align}\label{setM}
 H_i=\left(x_i-u_i+\sum_{j=1}^m \lambda_j\left[\partial_i\left(f_j\right)\right]\right)_w=x_i+\sum_{k \in S_i} \lambda_k\left[\partial_i\left(f_k\right)\right]_w =x_i+\sum_{k \in S} \lambda_k \partial_i\left(f_k\right)_w.
 \end{align}
 For $k \in S$, by the Euler relation \eqref{eulerrelation} we have 
 $$h_k\left(f_k\right)_w=\sum_{i \in I} w_i x_i \partial_i\left(f_k\right)_w$$
  (remark that  $\left(f_k\right)_w \in \mathbb{C}\left[x_I\right]$ by the previous observation).
 Multiplying $w_i x_i$ into functions \eqref{setK} and \eqref{setM}, it implies from the system of face functions \eqref{facesystem} that
$$
\begin{aligned}
& \sum_{k \in S} \lambda_k  w_i  x_i  \partial_i\left(f_k\right)_w=0 \text {, for } i \in K, \\
& \sum_{k \in S} \lambda_k  w_i x_i  \partial_i\left(f_k\right)_w+w_i  x_i^2=0 \text {, for } i \in M .
\end{aligned}
$$
Take the sum over all $i \in I$, we obtain
$$
\sum_{i \in M}w_i  x_i^2+\sum_{k \in S} h_k\left(f_k\right)_w=0 .
$$
Since $\left(f_k\right)_w=0$ we get
 $\sum_{i \in M}w_i x_i^2=0.$
Moreover, $w_i=\frac{e}{2}<0$ for $i \in M$ so $\sum_{i \in M} x_i^2=0.$
Denote by $Q(x_I)$ the quadratic polynomial $Q(x_I):=\sum_{i \in M} x_i^2.$

Thus, if the system \eqref{facesystem} has solution in $\left(\mathbb{C}^\times\right)^n \times\left(\mathbb{C}^\times\right)^m$ then the vectors  $\operatorname{grad} Q\left(x_{I}\right)$ and $\grad\left(f_k\right)_w(k \in S)$ are linearly dependent.
It means that the intersection of $\left\{Q\left(x_I\right)=0\right\}$ and $\left\{\left(f_k\right)_w(x)=0, k \in S\right\}$ is not transversal.
This is not true when $f_j$ are general.

 \textit{Case 2.2}: For any $i \in I, w_i \geq 0$. For $k \in S$, since $(f_k)_w\in \C[x_I]$, if the vector $a=(a_1, \ldots, a_n)$ belongs to the Newton polytope $ (P_k)_w$ of $(f_k)_w$ we have $a_i=0$ for $i\notin I$. Then 
 $$0\geq h_k= \langle w, a\rangle= \langle w_I, a_I\rangle\geq 0$$ for some $a\in (P_k)_w$. Therefore 
\begin{equation}\label{hk}
 h_k=0.
\end{equation} 
  By the same argument as in the proof of Theorem 7 in \cite[Page 13, Case 2]{BSW}, we also get that $w_i=0$ for all $i\in I$.
So $ h_k^i=e-w_i-v_k=h_k-w_i=0.$
This implies that the only possibilities are \eqref{3}, \eqref{4}, \eqref{7}. Remark that, due to \eqref{formulare}, for $i\in I$, the cases \eqref{4} and \eqref{7} can not occur simultaneously. We consider the following cases.

i). For some $i \in I$ we have \eqref{3}.
Then $e_i<\min \left\{0, w_i\right\}$, by \eqref{formulare} we get  $e_i=e-w_i=e$
so $e_i<0$ and $e<0$.
It deduces that for all $i \in I, e_i=e<0$. This means \eqref{3} occurs for all $i\in I$. Then we have the face function 
$$H_i=\sum_{k \in S_i} \lambda_k\left[\partial_i\left(f_k\right)\right]_w=\sum_{k \in S} \lambda_k \partial_i\left(f_k\right)_w.$$
If the system \eqref{facesystem} has solution in $\left(\mathbb{C}^\times\right)^n \times\left(\mathbb{C}^\times\right)^m$ then the tupe $((f_k)_w, k \in S)$ is  degenerate. This will not happens once $f_j$ are general.

ii). For all $ i \in I$ we have \eqref{4}.
Then $x_i=u_i$ and  $\left(f_k\right)_w(x)=0$. Remark that $ \left(f_k\right)_w \in \mathbb{C}\left[x_I\right].$
Thus $\left(f_k\right)_w\left(u_I\right)=0 \quad(k \in S)$.
This is not true once we choose $u$ generic.

iii). For all $i \in I$ we have \eqref{7}. It means $w_i=e_i=0$, combining with \eqref{formulare} one obtains that 
\begin{equation}\label{e=0}
e=0.
\end{equation}
In addition, for $k\in S$ we have $e=v_k+h_k$, combining with \eqref{hk} one gets $v_k=0.$ Therefore 
\begin{equation}\label{S}
S=\{k\in \{1, \ldots, m\} \mid h_k=v_k=0\}.
\end{equation} 

Furthermore, for all $ i \in I$ we have \eqref{7} leading that $$H_i=x_i-u_i+\sum_{k \in S} \lambda_k  \partial_i\left(f_k\right)_w.$$
Thus, we have a subsystem:
\begin{align}\label{subsystem}
\left\{\begin{array}{l}
\left(f_k\right)_w=0, k \in S, \\
x_i-u_i+\sum_{k \in S} \lambda_k \partial_i\left(f_k\right)_w=0, i \in I .
\end{array}\right.
\end{align}
This is a system in variables $x_i,i\in I$, its solutions are critical points of the distance function from $u_I=(u_i, i\in I)$ to the set $\left\{x_I:\left(f_k\right)_w\left(x_I\right)=0, k \in S\right\}$. To solve \eqref{facesystem}, we first solve \eqref{subsystem} and then consider the face functions $H_i$ for $i\notin I$.

Write 
$$ \{1, \ldots, n\} = I \sqcup J.$$
 For  $j \in \{1,\ldots, m\}$ and $ i \in J$, by \eqref{inequalityhij} we get $$v_j+h_j^i \geq v_j+h_j-w_i \geq e-w_i=-w_i,$$
the equality does not hold due to $i\notin I$. It implies that $ e_i>-w_i$ for all $i \in J$.
Thus, for  $i \in J$, the only possibilities are \eqref{3},\eqref{4},\eqref{5}. We consider two cases as follows.

 iii-a). For some $i \in J$ we have \eqref{3}. For such element $i$,  $e_i<\min \left\{0, w_i\right\}$, combining this with the fact $ e_i>-w_i$, one obtains that $w_i>0$. Thus $w_i \geq 0$  for all $i \in J$.
In particular $$h_j^i=\min _{a \in P_j^i} \langle w, a\rangle \geq 0\quad \textrm{for all}\quad j \in \{1,\ldots, m\}.$$
Take one element $i_0 \in J$ that \eqref{3} holds, i.e. $e_{i_0}<\min \left\{0, w_{i_0}\right\}$.
For $j \in S_{i_0}$, we have $0> e_{i_0}=v_j+h_j^{i_0} \geq v_j$. Consequently $h_j>v_j+h_j \geq e=0,$ contradicts to \eqref{ordernegative}. Thus \eqref{3} does not hold for any $i \in J$.

iii-b). For any $i \in J$, either \eqref{4} or \eqref{5} holds. Denote $K:=\left\{k \in J \mid w_k=0\right\}$ and $M:=\left\{m \in J \mid w_m>0\right\}.$ 
If $M=\emptyset$ then $w_i=0$ for all $i=1, \ldots, n$ implying that $h_j^i= h_j=0$ for all $j\in \{1,\ldots,  m\}$, all $i\in\{1, \ldots, n\}$. In particular, for any $i\in J$, the case \eqref{4} occurs and we  have $e_i>0$. In addition, according to \eqref{e=0}, we have
$$e_i=\min _{j=\overline{1, m}}\left\{v_j+h_j^i\right\}=\min _{j= \overline{1, m}} v_j=\min _{j=\overline{1, m}}\left\{v_j+h_j\right\}=e=0.$$
This is a contradiction. Hence $M \neq \emptyset$.

 Let $m \in M,$ then $w_m>0$, by \eqref{5} we get $e_m=0$.
For each $j \in S_m$, we have 
$e_m=v_j+h_j^m$  and $v_j \geq v_j+h_j \geq e= 0$ (due to \eqref{ordernegative} and \eqref{e=0}).
Since $w_i\geq 0$ for all $i=1, \ldots, n$, we obtain that 
$h_j^m=\min _{a \in P_j^m}\langle w, a\rangle \geq 0$, hence $v_j \leq v_j+h_j^m=e_m=0.$
Therefore, 
$$v_j= h_j=h_j^m=0\quad\textrm{and}\quad j\in S.$$
Thus,
$h_j^m=\langle w, a\rangle=0$ for any $ a \in\left(P_j^m\right)_w,$ combining this with  the fact  $w_i>0$ for all $i\in M$, one obtains that $a_i=0$ for any $i\in M$ and any  $ a \in\left(P_j^m\right)_w.$
It means that $\left(\partial_m\left(f_j\right)\right)_w$ does not depend on $x_i$ for $i \in M$. 

Now, one can rewrite part of the system \eqref{facesystem}  for $i\in J$ as follows.
 \begin{align}\label{subsystem2}
\left\{\begin{array}{l}
x_i-u_i=0, i \in K, \\
H_i=u_i-\sum_{k \in S_i} \lambda_k \left[\partial_i\left(f_k\right)\right]_w=0, i \in M.
\end{array}\right.
\end{align}
Let $(x_I, \lambda_k, k\in S)$ be a solution to the system \eqref{subsystem} for critical points of the distance function from $u_I$ to the set $\left\{x_I:\left(f_k\right)_w\left(x_I\right)=0, k \in S\right\}$. From the system \eqref{subsystem2}, for $i\in K$ we get $x_i=u_i$, for $i\in M$ the functions $H_i$ depend on $x_i$, $i\in I\sqcup K$ and $\lambda_k, k\in S_i\subset S$. For each $m\in M$, substituting these values into the polynomial $\sum_{k \in S_m} \lambda_k \left[\partial_m\left(f_k\right)\right]_w$ gives a constant which is not equal to $u_m$ for generic $u_m$. This completes the proof.

\end{proof}

\end{document}